\documentclass[letterpaper, reqno, 11pt]{amsart}
\usepackage[english]{babel}
\usepackage{graphicx,color}
\usepackage[all]{xy}
\usepackage{amsfonts,dsfont,mathrsfs}
\usepackage{enumerate}
\usepackage{amsmath,amscd}
\usepackage{amsmath}
\usepackage{amssymb}
\usepackage{amsthm}
\usepackage{tikz} 
\usepackage{tikz-cd}
\usepackage{graphics}
\usepackage{listings}
\usepackage{xcolor}
\usepackage{gensymb}
\usepackage{amssymb}
\usepackage{setspace}
\usepackage{fancyhdr}
\usepackage{dsfont}
\usepackage{bm}
\usepackage{enumitem}
\numberwithin{equation}{section}

\newtheorem{thm}{Theorem}

\newtheorem{prop}{Proposition}

\newtheorem{lem}{Lemma}

\newtheorem{cor}{Corollary}

\theoremstyle{definition}
\newtheorem{dfn}{Definition}

\newtheorem*{rmk}{Remark}

\newtheorem{conj}{Conjecture}

\newtheorem*{notation}{Notation}

\thanks{}

\begin{document}

\title{Connection Blocking in $\textrm{SL}(n,\mathds{R})$ Quotients}
\author[R. Bidar]{Mohammadreza Bidar}
\address{Department of Mathematics\\
                 Michigan State University\\
                 East Lansing, MI 48824}

\date{\today}

\begin{abstract}
Let $G$ be a connected Lie group and $\Gamma \subset G$ a lattice. Connection curves of the homogeneous space $M=G/\Gamma$ are the orbits of one parameter subgroups of $G$. To \textit{block} a pair of points $m_1,m_2 \in M$ is to find a \textit{finite} set $B \subset M\setminus \{m_1, m_2 \}$ such that every connecting curve joining $m_1$ and $m_2$ intersects $B$. The homogeneous space $M$ is \textit{blockable} if every pair of points in $M$ can be blocked.
\par
In this paper we investigate blocking properties of $M_n= \textrm{SL}(n,\mathds{R})/\Gamma$, where $\Gamma=\textrm{SL}(n,\mathds{Z})$ is the integer lattice. We focus on $M_2$ and show that the set of bloackable pairs is a dense subset of $M_2 \times M_2$, and we conclude manifolds $M_n$ are not blockable. Finally, we review a quaternionic structure of $\textrm{SL}(2,\mathds{R})$ and a way for making co-compact lattices in this context. We show that the obtained quotient homogeneous spaces are not finitely blockable. 
\end{abstract}

\maketitle

%\tableofcontents

\section{Introduction}

Finite blocking is an interesting concept originating as a problem in billiard dynamics and later in the context of Riemannian manifolds. Let $(M,g)$ be a complete connected, infinitely differentiable Riemannian manifold. For a pair of (not necessarily distinct) points $m_1,m_2 \in M$ let $\Gamma(m_1,m_2)$ be the set of geodesic segments joining these points. A set $B \subset M\setminus \{m_1,m_2\}$ is \textit{blocking} if every $\gamma \in \Gamma(m_1,m_2)$ intersects $B$. The pair $m_1,m_2$ is secure if there is a \textit{finite blocking} set $B=B(m_1,m_2)$. A manifold is secure if all pairs of points are secure. If there is a uniform bound on the cardinalities of blocking sets, the manifold is \textit{uniformly secure} and the best possible bound is the \textit{blocking number}. 
\par
Now, the first question naturally arising is what Riemannian manifolds are secure. If we focus on closed Riemannian manifolds, there is the following conjecture \cite{Growth of the number of geodesics,Blocking light}:
\newline
\begin{conj} A closed Riemannian manifold is secure if and only if it is flat.
\end{conj}
\par 
Flat manifolds are uniformly secure, and the blocking number depends only on their dimension \cite{Connecting geodesics and security, Blocking of billiard}. They are also \textit{midpoint secure}, i.e., the midpoints of connecting geodesics yield a finite blocking set for any pair of points \cite{Connecting geodesics and security, Insecurity for compact,Blocking of billiard}. Conjecture 1. says that flat manifolds are the only secure manifolds. This has been verified for several special cases: A manifold without conjugate points is uniformly secure if and only if it is flat \cite{Growth of the number of geodesics,Blocking light}; a compact locally symmetric space is secure if and only if it is flat \cite{Connecting geodesics and security}; the generic manifold is insecure \cite{dense G-delta,Real Analytic metrics,Insecurity is generic}; Conjecture 1. holds for compact Riemannian surfaces with genus bigger or equal than 1 \cite{Insecurity for compact}; any Riemannian metric has an arbitrarily close, insecure metric in the same conformal class \cite{Insecurity is generic}.  
\par
Gutkin \cite{Connection blocking} initiated the study of blocking properties of homogeneous spaces. Here, connection curves are the orbits of one-parameter subgroups of $G$. In this context, he speaks of \textit{finite blocking} instead of security; the counterpart of "secure" in this context is the term \textit{connection blockable}, or simply \textit{blockable}. A counterpart of Conjecture 1 for homogeneous spaces is as follows:
\newline
\begin{conj} 
Let $M=G/\Gamma$ where where $G$ is a connected Lie group and $\Gamma \subset G$ is a lattice. Then $M$ is blockable if and only if $G=\mathds{R}^n$, i.e., $M$ is a torus.
\end{conj}
\par
Gutkin in \cite{Connection blocking} establishes Conjecture 2 for nilmanifolds. He then proves the homogeneous space $\textrm{SL}(n,\mathds{R})/\textrm{SL}(n,\mathds{Z})$ is not \textit{midpoint blockable}. We continue his work by
proving that these spaces are not blockable. Specifically we prove:

\begin{thm}\label{thm-main}
Let $M_n= \textrm{SL}(n,\mathds{R})/\Gamma,\  \Gamma=\textrm{SL}(n,\mathds{Z})$. Two elements $m_1=g_1\Gamma$ and $m_2=g_2\Gamma \in M_2$ are not finitely blockable from each other if $g_1^{-1}g_2 \in \textrm{SL}(2,\mathds{Q})$. In particular, the set of non-blockable pairs is a dense subset of $M_2\times M_2$. 
\end{thm}

As we will see, this easily implies the following:

\begin{thm}\label{thm-Mn}
The homogeneous space $M_n,\, n>2$ has infinitely many pairs of non-blockable points. 
\end{thm}

\begin{rmk}
As discussed in Section 2, quotient spaces of a Lie group mod two commensurable lattices carry the same blocking property.
Margulis Arithmeticity Theorem \cite[p.92]{Arithmetic Groups}, \cite[p.298]{Discrete subgroups}, implies every lattice of $\textrm{SL}(n,\mathds{R}),\, n \geq 3$ is arithmetic. As a result, a large class of lattices in $\textrm{SL}(n,\mathds{R}),\, n \geq 3$ are commensurable to $\textrm{SL}(n,\mathds{Z})$. In particular, if $\Gamma$ is a lattice and the subgroup $\Gamma \cap \textrm{SL}(n,\mathds{Z})$ is of finite index in $\Gamma$, then $\Gamma$ and $\textrm{SL}(n,\mathds{Z})$ are commensurable. Hence, all the homogeneous spaces $\textrm{SL}(n,\mathds{R})/\Gamma$, for such lattices $\Gamma$ are non-blockable. Moreover, for every lattice $\Gamma \subset \textrm{SL}(n,\mathds{Q})$, $\textrm{SL}(n,\mathds{R})/\Gamma$ is non-blockable \cite[p.319]{Discrete subgroups}.
\end{rmk}

We also show that for a large class of cocompact lattices $\Gamma \subset \textrm{SL}(2,\mathds{R})$, $\textrm{SL}(2,\mathds{R})/\Gamma$ is not blockable. We specifically prove (see Section 4 for details):

\begin{thm}\label{thm-mainco}
Let $a,b$ be positive integers such that $\Gamma=\textrm{SL}(1,\mathds{H}_{\mathds{Z}}^{a,b})$ is a cocompact lattice of $G=\textrm{SL}(1,\mathds{H}_{\mathds{R}}^{a,b})$. If $g=x+yi \in \textrm{SL}(1,\mathds{H}_{\mathds{Q}}^{a,b})$, then $g\Gamma \subset G/\Gamma$ is not finitely blockable from $m_0=\Gamma$. Therefore the homogeneous space $G/\Gamma$ is not finitely blockable.
\end{thm}

\par
The organization of the paper is as follows. In Section 2, we review connection blocking concept and general properties for homogeneous space; then we formulate one parameter subgroups in the Lie group $\textrm{SL}(2,\mathds{R})$. In Section 3, we first prove a technical proposition, then we state and prove Theorem 1. Finally, we prove Theorem 2 concluding homogeneous spaces $\textrm{SL}(n,\mathds{R})/\textrm{SL}(n,\mathds{Z})$ are not blockable. In section 4, we also present a quaternionic structure of $\textrm{SL}(2,\mathds{R})$ and a way for making co-compact lattices in this context. We state and prove a counterpart of the technical proposition of Section 3, then we prove Theorem 3.

%We are pleased to thank Ralf Spatzier for helpful discussions. We would like to 
%thank Xiangdong Xie, who first told us of Bonk and Schramm's work. We would also like to thank Marc Bourdon for 
%informing us of his joint work with Kleiner \cite{BK}. We would like to thank Jeremy Tyson for helpful remarks concerning the relation between the different notions of quasicircles (see Remarks \ref{rmk:warning} and \ref{rmk:quasicircles}).

%Part of this work was completed during a collaborative
%visit of the third author to Ohio State University (OSU), which was partially funded by the Mathematics Research Institute at OSU.

%%%%%%%%%%%%%%%%%%%%%%%%%%
%%%%%%%%%%%%%%%%%%%%%%%%%%
%%%%%%%%%%%%%%%%%%%%%%%%%%

\section{Preliminaries}

In this section we review general preliminaries of blocking properties in homogeneous spaces, and the tools needed to state and prove our main result. We follow the notation and discussion in \cite{Connection blocking}.

\subsection{Connection blocking in homogeneous spaces}
Let $G$ be a connected Lie group, $M=G/\Gamma$, $\Gamma \subset G$ a lattice. For $g \in G, m \in M$, $g \cdot m$ denotes the action of $G$ on $M$. Let $\mathfrak{G}$ be the Lie Algebra of $G$ and let $\exp : \mathfrak{G} \rightarrow G$ be the exponential map. For $m_1,m_2 \in M$ let $C_{m_1,m_2}$ be the set of parametrized curves $c(t)=\exp(tx)\cdot m, 0 \leq t \leq 1$, such that $c(0)=m_1, c(1)=m_2$. We say that $C_{m_1,m_2}$ is the collection of \textit{connecting curves} for the pair $m_1,m_2$. Let $I \subset \mathds{R}$ be any interval. If $c(t), t\in I$, is a curve, we denote by $c(I)\subset M$ the set $\{c(t): t\in I\}$. A \textit{finite set} $B \subset M \ \{m_1,m_2 \}$ is a \textit{blocking set} for the pair $m_1,m_2$ if for any curve $c$ in $C_{m_1,m_2}$ we have $c([0,1])\cup B \neq \emptyset$.
If a blocking set exists, the pair $m_1,m_2$ is \textit{connection blockable}, or simply \textit{blockable}. 
The analogy with Riemannian security \cite{Blocking of billiard,Blocking light,Blocking: new examples,Insecurity for compact: commentary} suggests the following:

\begin{dfn} Let $M=G/\Gamma$ be a homogeneous space.  
\end{dfn}
\begin{enumerate}[label=\roman*)]
\item  $M$ is \textit{connection blockable} if every pair of its points is blockable. If there exists at least one non-blockable pair of points in $M$, then $M$ is non-blockable. 
\item $M$ is \textit{uniformly blockable} if there exists $N \in \mathds{N}$ such that every pair of its points can be blocked with a set $B$ of cardinality at most $N$. The smallest such $N$ is the blocking number for $M$. 
\item A pair $m_1,m_2 \in M$ is \textit{midpoint blockable} if the set $\{ c(1/2) : c \in C_{m_1,m_2} \}$ is finite. A homogeneous space is midpoint blockable if all pairs of its points are midpoints blockable. 
\item A homogeneous space is \textit{totally non-blockable} if no pair of its points is blockable.
\end{enumerate}

Blocking property of homogeneous spaces carries some straightforward and expected properties which can be summarized in the following proposition.

\begin{prop}\label{prop:genblock}
Let $M=G/\Gamma$ where $\Gamma \subset G$ is a lattice, and let $m_0=\Gamma$ be the identity element of $M$. Then the following holds:
\begin{enumerate}[label=\roman*)]
\item The homogeneous space $M$ is blockable (resp. uniformly blockable, midpoint blockable) if and only if all pairs $m_0,m$ are blockable (resp. uniformly blockable,midpoint blockable). The space $M$ is totally non-blockable if and only if no pair $m_0,m$ is blockable;
\item Let $\tilde{\Gamma} \subset \Gamma$ be lattices in $G$, let $M=G/\Gamma, \tilde{M}=G/\tilde{\Gamma}$, and let $p: \tilde{M} \rightarrow M$ be the covering. Let $m_1,m_2 \in M$ and let $\tilde{m_1},\tilde{m_2} \in \tilde{M}$ be such that $m_1=p(\tilde{m_1}),m_2=p(\tilde{m_2})$. If $B \subset M$ is a blocking set for $m_1,m_2$ (resp. $\tilde{B} \subset \tilde{M}$ is a blocking set for $\tilde{m_1},\tilde{m_2}$) then $p^{-1}(B)$ (resp. $p(\tilde{B}$) is a blocking set for $\tilde{m_1},\tilde{m_2}$ (resp. $m_1,m_2$).
\item Let $G', G''$ be connected Lie groups with lattices $\Gamma' \subset G',\Gamma'' \subset G''$, and let $M'=G'/\Gamma',M''=G''/\Gamma''$. Set $G=G'\times G'', M=M' \times M''$. Then a pair $(m_1',m_1''),(m_2',m_2'') \in M$ is connection blockable if and only if both pairs $m_1',m_2' \in M'$ and $m_1'',m_2'' \in M''$ are connection blockable.
\end{enumerate}
\end{prop}
\begin{proof}
Claim \textit{i)} is immediate from the definitions. The proofs of claim \textit{ii)} and claim \textit{iii)} are analogous to the proof of their counterparts for riemannian security. See prposition 1 in \cite{Connecting geodesics and security} for claim \textit{ii)}, and Lemma 5.1 and proposition 5.2 in \cite{Growth of the number of geodesics} for claim \textit{iii)}.
\end{proof}
We say homogeneous spaces $M_1,M_2$ have \textit{identical blocking property} if both are blockable (or not), midpoint blockable (or not), totally non-blockable (or not), etc. 

Recall that two subgroups $\Gamma_1,\Gamma_2 \subset G$ are \textit{commensurable}, $\Gamma_1 \sim \Gamma_2$, if there exists $g \in G$ such that the group $\Gamma_1 \cap g\Gamma_2g^{-1}$ has finite index in both $\Gamma_1$ and $g\Gamma_2g^{-1}$. Commensurability yields an equivalence relation in the set of lattices in $G$. We will use the following immediate Corollary of Proposition ~\ref{prop:genblock}.
\begin{cor}\label{cor:commens}

If lattices $\Gamma_1,\Gamma_2 \subset G$ are commensurable, then the homogeneous spaces $M_i=G/\Gamma_i : i=1,2$ have identical blocking properties.

\end{cor}

Let $\textrm{exp}:\mathfrak{G} \rightarrow G$ be the exponential map. For $\Gamma \subset G$ denote by $p_{\Gamma}:G \rightarrow G/\Gamma$ the projection, and set $\textrm{exp}_{\Gamma}=p_{\Gamma} \circ \textrm{exp}: \mathfrak{G} \rightarrow G/\Gamma$. We will say that a pair $(G,\Gamma)$ is of \textit{exponential type} if the map $\textrm{exp}_{\Gamma}$ is surjective. Let $M=G/\Gamma$. For $m \in M$ set $\textrm{Log}(m)=\textrm{exp}_{\Gamma}^{-1}(m)$. Note, $\textrm{Log}(m)$ may have more than one element. We will use the following basic fact to prove a point is not blockable from identity. See \cite{Connection blocking} Proposition 2 for the proof.

\begin{prop}\label{prop:fcosets}
Let $G$ be a Lie group, $\Gamma \subset G$ a lattice such that $(G,\Gamma)$ is of exponential type, and let $M=G/\Gamma$. 
Then $m\in M$ is blockable away from $m_0$ if and only if there is a map $x \mapsto t_x$ of $Log(m)$ to $(0,1)$ such that the set $\{\exp(t_x,x): x \in Log(m) \}$ is contained in a finite union of $\Gamma$-cosets. 
\end{prop}

The following lemma relates blocking property of a homogeneous space and its closed subspaces. The proof is straightforward and is left to the reader. 
\begin{lem}
Let $G$ be a Lie group, and let $\Gamma \subset G$ be a lattice. Let $H \subset G$ be a closed subgroup such that $\Gamma\cup H$ is a lattice in $H$. Let $X=G/\Gamma, Y=H/(\Gamma\cup H)$ be the homogeneous spaces, and let $Y\subset X$ be the natural inclusion. 
\begin{enumerate}[label=\roman*)]
\item If $Y$ is not blockable (resp. not midpoint blockable, etc) then $X$ is not blockable (resp. not midpoint blockable, etc).
\item If $Y$ contains a point which is not blockable (resp. not midpoint blockable) away from itself, then no point in $X$ is blockable (resp. not midpoint blockable) away from itself.
\end{enumerate}
\end{lem}

\subsection{One parameter families of $\textrm{SL}(2,\mathds{R})$ and modified times}

In this section we derive an explicit formula for one parameter families in $\textrm{SL}(2,\mathds{R})$, which is  essential to study its blocking properties.

The exponential map for $\textrm{SL}(2,\mathds{R})$ can be formulated in terms of trigonometric functions. The formula is directly derived from the exponential power series $\exp(X)=\sum_{k=0}^{\infty} X^k/k!$, doing some matrix algebra. For details see \cite[pp. 17-19]{Lie Groups-An Intro}. We have the following proposition:
\begin{prop}\label{prop:exponential} Let $g_0$ denote the identity element of $\textrm{SL}(2,\mathds{R})$. For a given matrix

\begin{equation}
X
=\begin{pmatrix} a & b \\
c &-a  \\
\end{pmatrix} \in \mathfrak{sl}(2,\mathds{R}),
\end{equation}

if $a^2+b>0$, then let $\omega(X)=\sqrt{a^2+bc}>0$ and if $a^2+b<0$, then let $\omega(X)=\sqrt{-(a^2+bc)}>0$. In the first case we have 
\begin{equation}\label{eq:expb}
\textrm{exp}(X)= (\cosh \omega) \, g_0 + \left(\dfrac{\sinh \omega}{\omega}\right) X,
\end{equation}
and in the second case $(a^2+bc<0)$, we have 
\begin{equation}
\textrm{exp}(X)= (\cos \omega) \, g_0 + \left(\dfrac{\sin \omega}{\omega}\right) X.
\end{equation}
If $a^2+bc=0$, then $\textrm{exp}(X)=g_0+X$. Furthermore, every matrix $g \in \textrm{SL}(2,\mathds{R})$ whose trace satisfies 
$tr(g) \geq -2$ is in the image of the exponential map. Consequently, for any $g \in \textrm{SL}(2,\mathds{R})$, either $g$ or $-g$ is of the form $\textrm{exp}(X)$, for some $X \in \mathfrak{sl}(2,\mathds{R})$. Therefore, $(\textrm{SL}(2,\mathds{R}),\textrm{SL}(2,\mathds{Z}))$ is of exponential type.
\end{prop}
\par
For $g \in G_2$ with $tr(g) \geq 2$, $\log(g)$ is unique and we use the notations
$\omega_g=\omega(\log(g)),\, g^t=\exp{(t\log g)},\, 0 \leq t \leq 1$. We have the following lemma:
\begin{lem}\label{lem-per}
If $tr(g) \geq 2$, we have

\begin{equation}
g^t=\left( \cosh(t\omega_g)-\dfrac{\sinh{t\omega_g}}{\sinh{\omega_g}}\cosh \omega_g \right)g_0+\dfrac{\sinh{t\omega_g}}{\sinh{\omega_g}}g, 
\end{equation} where $g_0$ is the identity element of $G_2$.
\end{lem}
\begin{proof}
From \eqref{eq:expb} it follows that 
\begin{equation}\label{eq:logg}
\log g= \dfrac{\omega_g}{\sinh \omega_g}\left( g-\cosh(\omega_g)g_0 \right)
\end{equation}
Noting that $\omega(t\log(g))=t\omega(\log(g))$, substituting \eqref{eq:logg} in the equation 
\begin{equation*}
g^t=\exp{(t\log g)}= \cosh(t\omega_g)g_0+\dfrac{\sinh{t\omega_g}}{t\omega_g}(t\log g)
\end{equation*}
gives the desired formula.

\end{proof}

For a fixed $g$ and an arbitrary $\gamma \in \Gamma$, as long as $g$ is known in the context, we use the notation $\lambda_{\gamma}=\dfrac{\sinh (t\omega_{g\gamma})}{\sinh (\omega_{g\gamma})},\ 0\ \leq \lambda_{\gamma} \leq 1$. We will call $\lambda_{\gamma}$'s modified times.
From \eqref{eq:expb} we have $\cosh \omega_{\gamma}=\textrm{tr}(g\gamma)/2$; a direct computation from \eqref{eq:logg} gives the following formula. 
\begin{equation}\label{ggammat}
(g\gamma)^t=\left[ a(\lambda_{\gamma})-1/2\textrm{tr}(g\gamma)\lambda_{\gamma} \right]g_0+\lambda_{\gamma}g\gamma, 
\end{equation} where 
\begin{equation}\label{alambdagamma}
a(\lambda_{\gamma})=\left( 1+ \left( \textrm{tr}(g\gamma)^2/4-1 \right)\lambda_{\gamma}^2  \right)^{1/2}
\end{equation}
Modified time as defined in above, will be pivotal for the proof of the main theorem. 

\begin{notation}
While working with a sequence $\{\gamma_{i}\} \in \Gamma$, by $\lambda_i, a(\lambda_i)$ we mean $\lambda_{\gamma_i}, a(\lambda_{\gamma_i})$. 
\end{notation}

\section{Blocking properties of $M_n$}

This section concludes with the proof of Theorem \ref{thm-main}. The proof will be based on the technical Proposition \ref{prop:lambda}, which is the main body of this section.  

\par
Throughout the section, $\Gamma=\textrm{SL}(2,\mathds{Z})$, $M_2=\textrm{SL}(2,\mathds{R})/\Gamma$. We assume:
$$g =\begin{pmatrix} x & y \\
z & w  \\ \end{pmatrix} \in \textrm{SL}(2,\mathds{Q}),\  \{ \gamma_{i} \} \subset \Gamma,\  g\gamma_{i}= \begin{pmatrix} x_{i} & y_{i} \\
z_i & w_i  \\ \end{pmatrix}\, .$$ Moreover, since $g$ and $-g$ have identical blocking properties, we may assume $x>0$.

\par
In order to prove Proposition ~\ref{prop:lambda}, we first need a few Lemmas.

\begin{lem}\label{polynomial3}
Suppose that $R(x,y)\in \mathds{R}[x,y]$ has the form $R(x,y)=cx^n+P(x,y),\, n>0,\, c\neq 0$, where $P(x,y)$ is of degree of at most $n-1$ in $x$. Then given any sequence of positive real numbers $\{y_i\}$ such that $y_i \rightarrow \infty$, as $i \rightarrow \infty$, there exists an increasing function $f: \mathds{Z}^+ \rightarrow \mathds{Z}^+$ such that $R(y_{f(i)},y_{f(j)})\neq 0,\, \forall i\neq j$.
\end{lem}
\begin{proof}
Define $f: \mathds{Z}^+ \rightarrow \mathds{Z}^+$ inductively as follows. Set $f(1)=1$, and assuming 
$f(k)$ is defined, define $f(k+1)$ in the following way. The $k$-polynomials $R_1(x)=R(x,y_{f(1)}),\cdots,R_k(x)=R(x,y_{f(k)})$ are all degree $n$ in $x$. Choose $l$ large enough so that $R_1(y_l),\cdots,R_k(y_l)\neq 0$, and define $f(k+1)=l$.
\end{proof}

\begin{lem}\label{gamma-lin1}
Let $(g\gamma_{1})^{t_{1}},\cdots, (g\gamma_{n})^{t_{n}}$ be $\mathds{Z}$-linearly dependent, that is \begin{equation}\label{mgamma} \sum_{i=1}^n m_i (g\gamma_{i})^{t_{i}}=0,\ m_i\in \mathds{Z} \, . \end{equation} Then we have  $\sum_{i=1}^n m_i \lambda_i = 0$ and $\sum_{i=1}^n m_i a(\lambda_i)=0$.
\end{lem}
\begin{proof}

By \eqref{ggammat} we have 
$$(g\gamma_{i})^{t_{i}}= \begin{pmatrix} a(\lambda_i)+1/2 \lambda_i(x_i-w_i) & \lambda_iy_i \\
\lambda_iz_i &  a(\lambda_i)+1/2 \lambda_i(w_i-x_i) \\
\end{pmatrix}. $$ 
Now \eqref{mgamma} implies

\begin{equation}\label{mgamma1}
\sum_{i=1}^n m_i \lambda_i y_i= 0\, ,
\end{equation}

\begin{equation}\label{mgamma2}
\sum_{i=1}^n m_i a(\lambda_i)+1/2 \lambda_i(x_i-w_i)= 0\, ,
\end{equation}

\begin{equation}\label{mgamma3}
\sum_{i=1}^n m_i a(\lambda_i)+1/2 \lambda_i(w_i-x_i)= 0\, .
\end{equation}

Since $y_1=\cdots=y_n=x$, \eqref{mgamma1} immediately implies $\sum_{i=1}^n m_i \lambda_i = 0$. To obtain $\sum_{i=1}^n m_i a(\lambda_i)=0$, add \eqref{mgamma2} and \eqref{mgamma3}.

\end{proof}

\begin{lem}\label{gamma-lin2} For  a given element $g\in \textrm{SL}(2,\mathds{R})$:
\begin{enumerate}[label=\roman*)]
\item Every five elements of coset $g\Gamma$ are $\mathds{Z}$ linearly dependent. 

\item Let $g\gamma_1,\cdots, g\gamma_n,\, n \leq 4$, be $\mathds{Z}$(or $\mathds{Q}$)-linearly independent elements of $g\Gamma$. Then there exists a non-zero integer $m_0$ such that for every $g\gamma \in span_{\mathds{Q}}<g\gamma_1,\cdots,g\gamma_n>$, there exists $(m_1,\cdots, m_n)\in \mathds{Z}^n$ so that $\sum_{i=1}^n m_i (g\gamma_i)=m_0(g\gamma)$.
\end{enumerate}
\end{lem}
\begin{proof}
To prove the first part note that  $g\gamma \in \textrm{span}_{\mathds{Q}}<g\gamma_1,\cdots,g\gamma_n>$ if and only if $\gamma \in \textrm{span}_{\mathds{Q}}<\gamma_1,\cdots,\gamma_n>$; therefore we may assume $g\Gamma=\Gamma$. Considering $\Gamma$ as a subset of $\mathds{Q}$-vector space $\mathds{Q}^4$, immediately implies every five elements of it are $\mathds{Q}$ (and therefore  $\mathds{Z}$)-linearly dependent.
\par
Now we prove part \textit{ii)} of the Lemma. First a conventional notation; For $$\gamma=\begin{pmatrix} \gamma^{1} &  \gamma^{2}\\
\gamma^{3} & \gamma^{4}  \\ \end{pmatrix} \in \Gamma \, ,$$ 
define $[\gamma]=(\gamma^{1},\gamma^{2},\gamma^{3},\gamma^{4})^{T}$; moreover we use the notation $[a]=(a_1,\cdots,a_n)^{T}$, to denote an arbitrary element of $\mathds{R}^n$ as a $n \times 1$ matrix. Let $A=([\gamma_1] \cdots [\gamma_n])$. Note that $A$ is a $4\times n$ matrix  of rank $n$, thus there exists an invertible $n\times n$ submatrix $\tilde{A}$ consisting of rows, say, $i_1,\cdots,i_n$. Take an arbitrary element $\gamma \in \textrm{span}_{\mathds{Q}}<\gamma_1,\cdots,\gamma_n>$. Since $\tilde{A}^{-1}$ has rational entries, we can choose a fixed integer $m_0$ so that $m_0 \tilde{A}^{-1}$ has integer entries. Hence the linear equation $\tilde{A} [m]=m_0(\gamma^{i_1},\cdots,\gamma^{i_n})^{T}$ has a solution $[m]=(m_1,\cdots,m_n)^{T}\in \mathds{Z}^n$. For $1\leq j \leq 4,\, j \neq i_1,\cdots,i_n$, let $A_j$ denote the $j$-th row of $A$, and assume $A_j=\sum_{k=1}^n \alpha_{jk}A_{i_k}$. Since $\gamma \in \textrm{span}_{\mathds{Q}}<\gamma_1,\cdots,\gamma_n>$, there exists $[r] \in \mathds{Q}^n$ such that $[\gamma]=A[r]$. It follows that: 
\begin{equation*}
\gamma^{j}=A_j[r]=(\sum_{k=1}^n \alpha_{jk}A_{i_k})[r]=\sum_{k=1}^n \alpha_{jk}(A_{i_k}[r])=\sum_{k=1}^n \alpha_{jk}\gamma^{i_k}\, .
\end{equation*}
Hence we have:
\begin{equation*}\label{Ajm}
A_j [m] = (\sum_{k=1}^n \alpha_{jk}A_{i_k}) [m] = \sum_{k=1}^n \alpha_{jk}(A_{i_k}[m])
        = \sum_{k=1}^n \alpha_{jk} m_0\gamma^{i_k} = m_0\gamma^{j}\, .
\end{equation*}
Therefore we conclude $A [m]=m_0 [\gamma]$, that implies $m_0 \gamma=\sum_{i=1}^n m_i \gamma_i$.
\end{proof}

\begin{prop}\label{prop:lambda}
If $m=g\Gamma \in \textrm{SL}(2,\mathds{Q})/\Gamma$, $g=\begin{pmatrix} x & 0 \\
z & 1/x  \\ \end{pmatrix}$, is finitely blockable from identity $m_0$, then there exists a sequence $$\gamma_{i}= \begin{pmatrix} p_{i} & 1 \\
p_{i}s_{i}-1 & s_{i}  \\ 
\end{pmatrix} \in \Gamma$$ and a sequence of times $\{ t_{i} \} \subset (0,1)$ such that 
\begin{enumerate}[label=\roman*)]

\item all elements of $\{ (g\gamma_{i})^{t_i}\}$ belong to the same coset, and all modified times are the same, i.e., $\lambda_{i}=\lambda=\textrm{const}$,

\item $\lambda_{i}^2,\lambda_{i} a(\lambda_{i}) \in \mathds{Q}$,

\item $C_i=\textrm{tr}(g\gamma_{i})$ is an increasing sequence of positive rational numbers with the same denominator, $C_i \rightarrow \infty$, as $i \rightarrow \infty$, and

\item $\{p_{i} \}$ is an increasing sequence of positive integers. 
\end{enumerate}
\end{prop}

\begin{proof}
Let $m=g\Gamma,\,g \in \textrm{SL}(2,\mathds{Q})$, be blockable from identity $m_0$. Suppose $x=a/b,\ a,b \in \mathds{Z}$; let $C_i=z+ib$ and $p_{i}=2ib^2$, $s_{i}=(a-2a^2)i$. It is clear that $2<C_1<C_2<\cdots$ and $\textrm{tr}(g\gamma_{i})=C_i$. Note that $C_i$'s are rational numbers with the same denominator.  
By proposition 2.1. for a suitable choice of $t_{i}$'s where $0<t_{i}<1$ we should have $\{ (g\gamma_{i})^{t_{i}} \} \subset \cup_{n=1}^N \tilde{g_n}\Gamma$; passing to a subsequence if necessary it follows that there exists a sequence $\gamma_{i} \in \Gamma,\ \gamma_{i}= \begin{pmatrix} p_{i} & 1 \\
p_{i}s_{i}-1 & s_{i}  \\
\end{pmatrix}$ such that $tr(g\gamma_{i})=C_i=z+n_{i}b$, $p_{i}=2n_{i}b^2$  where $n_i \in \mathds{Z}^{+}$, $n_1<n_2<\cdots$and $(g\gamma_{i})^{t_{i}} \in \tilde{g}\Gamma$ for some fixed $\tilde{g} \in G$. Now let $\lambda_{i}=\dfrac{\sinh(t_i\omega_{\gamma_{i}})}{\sinh(\omega_{\gamma_{i}})}$ be modified times $\lambda_i \in (0,1)$. We show that for every pair of indexes $(i,j)$, $\lambda_i \lambda_j \in \mathds{Q}$. By \eqref{ggammat} we have 
$$(g\gamma_{i})^{t_{i}}= \begin{pmatrix} a(\lambda_i)+1/2 \lambda_i(x_i-w_i) & \lambda_iy_i \\
\lambda_iz_i &  a(\lambda_i)+1/2 \lambda_i(w_i-x_i) \\
\end{pmatrix}, $$ where

$$a(\lambda_i)=a(\lambda_{\gamma_i})=\left[  1+ \left( 1/4\textrm{tr}(g\gamma_i)^2-1 \right)\lambda_{i}^2   \right]^{1/2}.$$ Since $\left[ (g\gamma_i)^{t_i}\right]^{-1} \cdot \left[ (g\gamma_j)^{t_j}\right] \in \Gamma$ it follows that
\begin{equation*}
\begin{split}
 & \begin{pmatrix} a(\lambda_i)+1/2 \lambda_i(w_i-x_i) & -\lambda_iy_i \\
-\lambda_iz_i &  a(\lambda_i)+1/2 \lambda_i(x_i-w_i) \\
\end{pmatrix} \cdot \\
  & \qquad \begin{pmatrix} a(\lambda_j)+1/2 \lambda_j(x_j-w_j) & \lambda_jy_j \\
\lambda_jz_j &  a(\lambda_j)+1/2 \lambda_j(w_j-x_j) \\
\end{pmatrix} \in \Gamma
\end{split}
\end{equation*}
which can be written as 
\begin{equation}\label{A-det}
B(i,j) \begin{pmatrix}  a(\lambda_i)a(\lambda_j) \\
\lambda_i a(\lambda_j) \\
a(\lambda_i) \lambda_j \\
\lambda_i \lambda_j    \\
\end{pmatrix} \in \mathds{Z}^4
\end{equation}
where
\begin{equation*}
B(i,j)= \begin{pmatrix}
1 & 1/2(w_i-x_i) & 1/2(x_j-w_j) & 1/4(w_i-x_i)(x_j-w_j)-y_iz_j          \\
0 & -y_i & y_j   & 1/2y_j(w_i-x_i)-1/2y_i(w_j-x_j)
                             \\
0 & -z_i & z_j   & 1/2z_j(x_i-w_i)-1/2z_i(x_j-w_j)  
                             \\
1 &  1/2(x_i-w_i) & 1/2(w_j-x_j) & 1/4(x_i-w_i)(w_j-x_j)-z_iy_j           \\                                                                         
\end{pmatrix}
\end{equation*}

We claim that passing to a subsequence of $\{(g\gamma_{i})^{t_{i}}\}$ if necessary, we may assume $\det(B(i,j)) \neq 0$. Let $u_i=x_i-w_i$, then a direct but lengthy computation shows that 
\begin{equation*}
\det(B(i,j))=x\left( u_i^2z_j+u_j^2z_i-u_iu_j(z_i+z_j)-x(z_j-z_i)^2 \right)
\end{equation*}
Noting that $u_i=2xp_i-C_i=(4ab-b)n_i-z, z_i=-xp_i^2+C_ip_i-1/x=b^3(2-4a)n_i^2+2zb^2n_i-b/a$, we see that $\det(B(i,j))=-a^2b^4(2-4a)^2n_j^4+P(n_i,n_j)$ where $P$ is a third degree polynomial in $n_j$. Now Lemma \ref{polynomial3} proves the claim.

Now, from \eqref{A-det} $\lambda_i^2,\lambda_i\lambda_j,\lambda_i a(\lambda_j) \in \mathds{Q}$. Let $1 \leq n_0 \leq 4$ be the biggest integer such that there are $n_0$ $\mathds{Q}$ (or $\mathds{Z}$)-linearly independent elements of $(g\gamma_{i})^{t_{i}} \in \tilde{g}\Gamma$. Then it is clear that   $(g\gamma_{i})^{t_{i}} \in \textrm{span}_{\mathds{Q}}<(g\gamma_1)^{t_1},\cdots,(g\gamma_{n_0})^{t_{n_0}}>$, for arbitrary $i$. Lemma \ref{gamma-lin1} implies that 
\begin{equation*}
m_0 \lambda_{i}=m_1\lambda_1+\cdots+m_{n_0}\lambda_{n_0}
\end{equation*}
and since $(g\gamma_{i})^{t_{i}},(g\gamma_1)^{t_1},\cdots,(g\gamma_{n_0})^{t_{n_0}}$ all belong to the same coset, by Lemma \ref{gamma-lin2} we can assume $m_0$ is fixed and does not depend on $i$. Now, from previous step and the equation 
\begin{equation*}
m_0^2 \lambda_{i}^2=(m_1\lambda_1+\cdots+m_{n_0}\lambda_{n_0})^2
\end{equation*}
we conclude $\{ \lambda_{i} \}$ does not have any accumulation point and since $\{\lambda_{i}\} \subset (0,1)$ it follows that it's finite. Passing to a subsequence again, we may assume $\lambda_{i}=\lambda=const$.
\end{proof}

\begin{lem}\label{lem-cos}
Every coset $m \in \textrm{SL}(2,\mathds{Q})/\Gamma$ has a representative of the form 
$g= \begin{pmatrix} x & 0 \\
z & 1/x  \\
\end{pmatrix}$; that is $m=g\Gamma$, where $x,z \in \mathds{Q}$.
\end{lem}
\begin{proof}
Let $g_1= \begin{pmatrix} x & y \\
z & w  \\
\end{pmatrix}$ be an arbitrary representative of coset $m$. If $y \neq 0$, let $s/q=-x/y,\ \textrm{gcd}(s,q)=1$ and choose $p,r \in \mathds{Z}$ so that $ps-rq=1$, and let $\gamma= \begin{pmatrix} p & q \\
r & s  \\
\end{pmatrix}$. It is clear that $g=g_1\gamma \in m$ and $g$ has the desired form.
\end{proof}

\par
Now we are ready to prove Theorem \ref{thm-main}.

\begin{proof}[Proof of Theorem \ref{thm-main}]
By contrary suppose $m=g\Gamma \in \textrm{SL}(2,\mathds{Q})/\Gamma$ is blockable from identity $m_0$.
By Lemma \ref{lem-cos} we may assume: $$g=\begin{pmatrix} x & 0 \\
z & 1/x  \\ \end{pmatrix},\  x,z \in \mathds{Q}\, .$$
Let $\{ (g\gamma_{i})^{t_i}\}$ be a sequence as in proposition ~\ref{prop:lambda}, and suppose $\textrm{tr}(g\gamma_{i})=C_i=x_i/y, x_i,y \in \mathds{Z}^{+}$, and $\lambda_{i}^2=\lambda^2=k/l<1, k,l \in \mathds{Z}^{+}$. Substituting theses into \eqref{alambdagamma} it follows that
\begin{equation*}
\left( \lambda_ia(\lambda_i)\right)^2= \frac{1}{4y^2l^2}\left(4kly^2-4k^2 y^2+k^2x_i^2\right)\, .
\end{equation*}
By Proposition ~\ref{prop:lambda}, \textit{ii}), we have $\lambda_ia(\lambda_i)=\lambda a(\lambda_i)\in \mathds{Q}$, so $\left( \lambda_ia(\lambda_i)\right)^2= a_i^2/b_i^2$, for some $\ a_i,b_i \in \mathds{Z}^{+}$. 
Thus there exists $\tilde{a}_i \in \mathds{Z}^{+}$ so that
\begin{equation*}
\left( 4kl-4k^2 \right)y^2+k^2x_i^2=\tilde{a}_i^2\, ,
\end{equation*}
which can be rewritten as
\begin{equation*}
\left( 4kl-4k^2 \right)y^2=(\tilde{a}_i+k x_i)(\tilde{a}_i-k x_i)\, .
\end{equation*}
Since $k<l$, left side is a constant positive integer. Letting $x_i \rightarrow \infty$, the above equation yields a contradiction. 
\end{proof}

From above theorem it immediately follows:

\begin{cor}\label{BI}
Two elements $m_1=g_1\Gamma$ and $m_2=g_2\Gamma \in M_2$ are not blockable from each other if $g_1^{-1}g_2 \in \textrm{SL}(2,\mathds{Q})$, therefore the set of non-blackable pairs is a dense subset of $M_2\times M_2$.
\end{cor}
Following the proof of Proposition 9 in \cite{Connection blocking}, we prove Theorem \ref{thm-Mn}:

\begin{proof}[Proof of Theorem \ref{thm-Mn}]
For $1 \leq i \leq n-1$ let $G_i \subset \textrm{SL}(n,\mathds{R})$ be the group $\textrm{SL}(2,\mathds{R})$ embedded in $\textrm{SL}(n,\mathds{R})$ via the rows and columns $i,i+1$. Then $G_i\cap \textrm{SL}(n,\mathds{Z}) \cong \textrm{SL}(2,\mathds{Z})$, and hence $G_i\textrm{SL}(n,\mathds{Z})/\textrm{SL}(n,\mathds{Z})\cong\textrm{SL}(2,\mathds{R})/\textrm{SL}(2,\mathds{Z})$. Set $M_n^{(i)}=G_i\textrm{SL}(n,\mathds{Z})/\textrm{SL}(n,\mathds{Z}) \subset M_n$. By Theorem \ref{thm-main}, each $M_n^{(i)}$ has infinitely many non-blockable pairs $m_1,m_2$, yielding the the claim.
\end{proof}

\section{Blocking property and cocompact lattices of $\textrm{SL}(2,\mathds{R})$}
$\textrm{SL}(2,\mathds{Z})$ is the obvious example of a lattice which is not cocompact. Up to commensurability and conjugates, this is the only one that is not cocompact, Morris \cite[p.115]{Arithmetic Groups}. Since quotient spaces of the same Lie group mod conjugate or commensurable lattices have identical blocking property, we have shown for every non cocompact lattice $\Gamma$, $\textrm{SL}(2,\mathds{R})/\Gamma$ has a dense subset of points not finitely blockable from identity. 
\par
To address the cocomapct lattices, we need to know more about the structure of these lattices. There are several ways to construct cocompact lattices of $\textrm{SL}(2,\mathds{R})$. Here, we study blocking properties for a class of cocompact lattices, in $\textrm{SL}(2,\mathds{R})$, derived from quaternion algebras. We follow the notation and discussion used in Morris \cite[p.118]{Arithmetic Groups}. First we need a few preliminaries.
\begin{dfn}
\begin{enumerate}
\item For any field $F$, and any nonzero $a,b \in F$, the corresponding \textbf{quaternion algebra} over $F$ is the ring 
\begin{equation*}
\mathds{H}_F^{a,b}=\{x+yi+zj+wk\, |\, x,y,z,w \in F \},
\end{equation*}
where
\begin{itemize}
\item addition is defined in the obvious way, and 
\item multiplication is determined by the relations
$$i^2=a,\ j^2=b,\ ij=k=-ji,$$ together with the requirement that every element of $F$ is in the center of $\mathds{H}_F^{a,b}$. 
(Note that $k^2=k\cdot k=(-ji)(ij)=-ab$.)
\end{itemize}
\item The \textbf{reduced norm} of $g=x+yi+zj+wk \in \mathds{H}_F^{a,b}$ is 
$$\textrm{N}_{\textrm{red}}(g)= g\bar{g}= x^2-ay^2-bz^2+abw^2 \in F,$$
where $\bar{g}=x-yi-zj-wk$ is the \textbf{conjugate} of $g$. (Note that $\overline{gh}=\bar{g}\bar{h}$.)

\end{enumerate}
\end{dfn}
There are a few straightforward facts left to the reader to verify, for example: $\mathds{H}_F^{a^2,b} \cong \textrm{Mat}_{2\times 2}(F)$
for any nonzero $a,b \in F$, $\mathds{H}_{\mathds{C}}^{a,b} \cong \textrm{Mat}_{2\times 2}(\mathds{C})$.

We need the following proposition:

\begin{prop}
Fix positive integers $a$ and $b$, and let 
\begin{equation*}
G=\textrm{SL}(1,\mathds{H}_{\mathds{R}}^{a,b})=\{g\in \mathds{H}_{\mathds{R}}^{a,b} | \textrm{N}_{\textrm{red}}(g)=1 \}.
\end{equation*}
Then:
\begin{enumerate}[label=\roman*)]
\item $G \cong \textrm{SL}(2,\mathds{R})$,
\item $G_{\mathds{Z}}=\textrm{SL}(1,\mathds{H}_{\mathds{Z}}^{a,b})$ is an arithmetic subgroup of $G$, and 
\item the following are equivalent:
\begin{enumerate}
\item $G_{\mathds{Z}}$ is cocompact in $G$.
\item $(0,0,0,0)$ is the only integer solution $(p,q,r,s)$ of the Diophantine equation 
\begin{equation*}
w^2-ax^2-by^2+abz^2=0.
\end{equation*}
\item Every nonzero element of $\mathds{H}_{\mathds{Q}}^{a,b}$ has a multiplicative inverse (so $\mathds{H}_{\mathds{Q}}^{a,b}$ is
a "division algebra").
\end{enumerate}
\end{enumerate}
\end{prop}

\begin{rmk}
It is well known that the Diophantine equation $w^2-ax^2-by^2+abz^2=0$ has only trivial integer solution if and only if the equation 
$ax^2+by^2=z^2$ has only trivial integer solution \cite[p.121]{Arithmetic Groups}. This can happen if $a,b$ are prime, or if $a$ is 
not a square mod $b$, and $b$ is not a square mod $a$. Throughout the section we assume $a$ and $b$ are such integers, so the norm 
equation has only trivial solution (and thus $G_{\mathds{Z}}$ is cocompact). In particular, $a$ and $b$ can not be perfect squares.
\end{rmk}

We refer the reader to \cite[p.119]{Arithmetic Groups} for a proof. We will use the fact that the isomorphism  in \textit{i}) is given by:

\begin{equation}\label{quaternion-iso}
\phi(x+yi+zj+wk)=
\begin{pmatrix}
x+y\sqrt{a} & z+w\sqrt{a} \\
b(z-w\sqrt{a}) & x-y\sqrt{a}
\end{pmatrix}\,.
\end{equation}

Next, we discuss the exponential mapping. Let $\mathfrak{G}\cong T_1(\textrm{SL}(1,\mathds{H}_{\mathds{R}}^{a,b}))$ and $\mathfrak{sl}(2,\mathds{R})\cong T_{\textrm{Id}}\textrm{SL}(2,\mathds{R})$ be the lie algebras of $G$
and $\textrm{SL}(2,\mathds{R})$ respectively. Since $\phi$ in equation \eqref{quaternion-iso} is an isomorphism of lie groups, $\textrm{d}\phi_1: T_1(\textrm{SL}(1,\mathds{H}_{\mathds{R}}^{a,b})) \rightarrow T_{\textrm{Id}}\textrm{SL}(2,\mathds{R})$ is a Lie algebra isomorphism. Moreover, since $\textrm{SL}(1,\mathds{H}_{\mathds{R}}^{a,b})$ and $\textrm{SL}(2,\mathds{R})$ are embedded manifolds in $\mathds{R}^4$, $\textrm{d}\phi_1$ is the restriction of the corresponding differential when $\phi$ is regarded as a function from $\mathds{R}^4$ to $\mathds{R}^4$. Note that $T_1(\textrm{SL}(1,\mathds{H}_{\mathds{R}}^{a,b}))=\{(0,u_1,u_2,u_2) | u_1,u_2,u_3 \in \mathds{R} \}$, computing $\textrm{d}\phi_1$ it follows that:
\begin{equation}
\begin{pmatrix}
1 & \sqrt{a} & 0 & 0   \\
0 & 0 & 1 & \sqrt{a}   \\
0 & 0 & b & -b\sqrt{a} \\
1 & -\sqrt{a} & 0 & 0 
\end{pmatrix} \cdot 
\begin{pmatrix}
0   \\
u_1 \\
u_2 \\
u_3
\end{pmatrix} = 
\begin{pmatrix}
u_1\sqrt{a}       \\
u_2+u_3\sqrt{a}   \\
bu_2-b\sqrt{a}u_3 \\
-u_1\sqrt{a}
\end{pmatrix}
\end{equation}
Since the diagram 
\begin{equation}
\begin{tikzcd}
\mathfrak{G} \arrow{d}[left]{\exp} \arrow{r}{\textrm{d}\phi_1}
& \mathfrak{sl}(2,\mathds{R}) \arrow{d}{\exp} \\
G \arrow{r}{\phi}
& \textrm{SL}(2,\mathds{R})
\end{tikzcd}
\end{equation}
commutes Proposition \ref{prop:exponential} easily implies the following:
\begin{prop}
Let $G=\textrm{SL}(1,\mathds{H}_{\mathds{R}}^{a,b})$ and $\mathfrak{G} \cong \mathds{R}^3$ be its Lie algebra. Given $U=(u_1,u_2,u_3) \in \mathfrak{G}$, let $\omega=\sqrt{|u_1^2a+u_2^2b-u_3^2ab|}$. Then we have the following:
\begin{enumerate}[label=\roman*)]
\item $\exp(U)=\cosh\omega+\dfrac{\sinh\omega}{\omega}u_1i+\dfrac{\sinh\omega}{\omega}u_2j+\dfrac{\sinh\omega}{\omega}u_3k$,
if $u_1^2a+u_2^2b-u_3^2ab>0$,
\item $\exp(U)=1+u_1i+u_2j+u_3k$, if $u_1^2a+u_2^2b-u_3^2ab=0$, and
\item $\exp(U)=\cos\omega+\dfrac{\sin\omega}{\omega}u_1i+\dfrac{\sin\omega}{\omega}u_2j+\dfrac{\sin\omega}{\omega}u_3k$,
if $u_1^2a+u_2^2b-u_3^2ab<0$.
\end{enumerate}
\end{prop}
For $g=x+yi+zj+wk \in G$ with $x>1$, $\log(g)$ is unique; let
$\omega_g=\omega(\log(g)),\, g^t=\exp{(t\log g)},\, 0 \leq t \leq 1$. The following Lemma is the counterpart to Lemma \ref{lem-per} and is stated as follows:
\begin{lem}\label{lem-per1}
Let $g=x+yi+zj+wk \in G$ with $x>1$, we have:

\begin{equation}
g^t=\left( \cosh(t\omega_g)-\dfrac{\sinh{t\omega_g}}{\sinh{\omega_g}}\cosh \omega_g \right)1+\dfrac{\sinh{t\omega_g}}{\sinh{\omega_g}}g. 
\end{equation}
\end{lem}
\begin{proof}
Follow the steps of Lemma \ref{lem-per}.
\end{proof}
Let $\Gamma=\textrm{SL}(1,\mathds{H}_{\mathds{Z}}^{a,b})$ be a cocompact lattice. Following notations of Section 2, for a fixed $g$ and an arbitrary $\gamma \in \Gamma$, $\lambda_{\gamma}=\dfrac{\sinh (t\omega_{g\gamma})}{\sinh (\omega_{g\gamma})},\ 0\ \leq \lambda_{\gamma} \leq 1$ is the modified time. Through similar step we can easily conclude:

\begin{equation}\label{ggammat-co}
(g\gamma)^t=\left[ a(\lambda_{\gamma})-x\lambda_{\gamma} \right]1+\lambda_{\gamma}g\gamma, 
\end{equation} where 
\begin{equation}\label{alambda-co}
a(\lambda_{\gamma})=\left( 1+ \left(x^2-1 \right)\lambda_{\gamma}^2  \right)^{1/2}\, .
\end{equation}
To follow through the proof of Proposition~\ref{prop:lambda} for co-compact lattices we only  consider elements $g=x+yi \in \textrm{SL}(1,\mathds{H}_{\mathds{Q}}^{a,b})$. For a sequence $\{\gamma_i\} \subset \Gamma$ let $g\gamma_i=x_i+y_ii+z_ij+w_ik$. We need the following lemma.
\begin{lem}\label{gamma-i-co}
Let $g=x+yi \in \textrm{SL}(1,\mathds{H}_{\mathds{Q}}^{a,b})$. There exists a sequence $\gamma_i=p_i+q_ii+r_ij+s_ik \in \Gamma$, such that $z_i$ and $w_i$ in $g\gamma_i$, are fixed for all $i$, $z_i^2-aw_i^2 \neq 0$, and $x_i \rightarrow \infty$, as $i \rightarrow \infty$.
\end{lem}
\begin{proof}
Fix an arbitrary element $\gamma_1=p_1+q_1i+r_1j+s_1k \in \Gamma$. Since $a$ is not a perfect square, $r_1^2-as_1^2 \neq 0$. Let $n=p_1^2-aq_1^2$. It is well known that if the Pell's equation $p^2-aq^2=n$ has one solution (and $a$ is not a perfect square), it has infinitely many solutions. Let ${(p_i,q_i)} \in \mathds{Z}^2$ be an infinite set of distinct solutions such that $xp_i,\ yq_i >0$, and let $\gamma_i=p_i+q_ii+r_1j+s_1k$. Then it is easily seen $z_i=xr_1+ays_1, w_i=xs_1+yr_1$ are fixed, $z_i^2-aw_i^2=(x^2-ay^2)(r_1^2-as_1^2)\neq 0$, and $x_i=xp_i+ayq_i \rightarrow \infty$ as
$i \rightarrow \infty$. 
\end{proof}
It can be easily seen Lemma \ref{gamma-lin2} is valid for the co-compact lattices $\Gamma$, if we think of elements of $\Gamma$ as two by two matrices with integer entries. The following proposition is the counterpart to Proposition \ref{prop:lambda} for co-compact lattices.

\begin{prop}\label{prop:lambda2}
Let $g=x+yi \in \textrm{SL}(1,\mathds{H}_{\mathds{Q}}^{a,b})$. $m=g\Gamma$, is finitely blockable from identity $m_0$, then there exists a sequence $\gamma_i=p_i+q_ii+r_ij+s_ik$ and a sequence of times $\{ t_{i} \} \subset (0,1)$ such that 
\begin{enumerate}[label=\roman*)]

\item all elements of $\{ (g\gamma_{i})^{t_i}\}$ belong to the same coset, and all modified times are the same, i.e., $\lambda_{i}=\lambda=\textrm{const}$,

\item $\lambda_{i}^2,\lambda_{i} a(\lambda_{i}) \in \mathds{Q}$,

\item $x_i=\textrm{Re}(g\gamma_{i})$ is an increasing sequence of positive rational numbers with the same denominator, $x_i \rightarrow \infty$, as $i \rightarrow \infty$, and

\item $\{p_{i} \}$ is an increasing sequence of positive integers. 

\end{enumerate}
\end{prop}
\begin{proof}
Let $m=g\Gamma$, be blockable from identity $m_0$. Let $\{\gamma_i\}$ be a sequence as in Lemma \ref{gamma-i-co}. Then $x_i=\textrm{Re}(x_i)$ is an increasing sequence of rational numbers with the same denominator, and $x_i \rightarrow \infty$, as $i \rightarrow \infty$.
By proposition \ref{prop:fcosets} for a suitable choice of $t_{i}$'s where $0<t_{i}<1$ we should have $\{ (g\gamma_{i})^{t_{i}} \} \subset \cup_{n=1}^N \tilde{g_n}\Gamma$; passing to a subsequence if necessary, we may assume $(g\gamma_{i})^{t_{i}} \in \tilde{g}\Gamma$ for some fixed $\tilde{g} \in G$. 

Now let $\lambda_{i}=\dfrac{\sinh(t_i\omega_{\gamma_{i}})}{\sinh(\omega_{\gamma_{i}})}$ be modified times $\lambda_i \in (0,1)$. We show that for every pair of indexes $(i,j)$, $\lambda_i \lambda_j \in \mathds{Q}$. 

By \eqref{ggammat-co} and \eqref{alambda-co} we have 
$$ (g\gamma_i)^{t_i}=a(\lambda_{i})+\lambda_{i}(y_ii+z_ij+w_ik) $$ where
$$a(\lambda_{i})=\left( 1+ \left(x^2-1 \right)\lambda_{i}^2  \right)^{1/2}\, .$$

Since $\left[ (g\gamma_i)^{t_i}\right]^{-1} \cdot \left[ (g\gamma_j)^{t_j}\right] \in \Gamma$ it follows that
\begin{equation*}
\left( a(\lambda_{i})-\lambda_{i}(y_ii+z_ij+w_ik) \right) \cdot \left( a(\lambda_{j})+\lambda_{j}(y_ji+z_jj+w_jk)\right) \in \Gamma
\end{equation*}
which can be written as 
\begin{equation}\label{A-det}
B(i,j) \begin{pmatrix}  a(\lambda_i)a(\lambda_j) \\
\lambda_i a(\lambda_j) \\
a(\lambda_i) \lambda_j \\
\lambda_i \lambda_j    \\
\end{pmatrix} \in \mathds{Z}^4
\end{equation}
where
\begin{equation*}
B(i,j)= \begin{pmatrix}
1 & 0    & 0     & w_iw_jab-y_iy_ja-z_iz_jb   \\
0 & -y_i & y_j   & (z_iw_j-w_iz_j)b
                             \\
0 & -z_i & z_j   & (w_iy_j-y_iw_j)a  
                             \\
0 & -w_i & w_j   & z_iy_j-y_iz_j             \\                                                                         
\end{pmatrix}
\end{equation*}

We claim that passing to a subsequence of $\{g\gamma_{i}^{t_{i}}\}$ if necessary, we may assume $\det(B(i,j)) \neq 0$. A direct computation shows that 
\begin{equation*}
\det(B(i,j))=-(w_iy_j-y_iw_j)^2a-(w_iz_j-z_iw_j)^2b+(z_iy_j-y_iz_j)^2
\end{equation*}

Note that $\det(B(i,j))$ is a second degree polynomial in $y_j$ (the coefficient of $y_j^2$ is $z_i^2-aw_i^2 \neq 0$); so by Lemma \ref{polynomial3} and passing to a subsequence if necessary, we may assume  $\det(B(i,j)) \neq 0$.
\par
Now, from \eqref{A-det} $\lambda_i^2,\lambda_i\lambda_j,\lambda_i a(\lambda_j) \in \mathds{Q}$. Let $1 \leq n_0 \leq 4$ be the biggest integer such that there are $n_0$ $\mathds{Q}$ (or $\mathds{Z}$)-linearly independent elements of $(g\gamma_{i})^{t_{i}} \in \tilde{g}\Gamma$. Then it is clear that   $(g\gamma_{i})^{t_{i}} \in \textrm{span}_{\mathds{Q}}<(g\gamma_1)^{t_1},\cdots,(g\gamma_{n_0})^{t_{n_0}}>$, for arbitrary $i$ which implies, considering the $z$-component, 
\begin{equation*}
m_0 \lambda_{i} z_i=m_1\lambda_1 z_1+\cdots+m_{n_0}\lambda_{n_0} z_{n_0}\, .
\end{equation*}
By Lemma \ref{gamma-i-co} $z_i$ is fixed and since $(g\gamma_{i})^{t_{i}},(g\gamma_1)^{t_1},\cdots,(g\gamma_{n_0})^{t_{n_0}}$ all belong to the same coset, by Lemma \ref{gamma-lin2} we can assume $m_0$ is also fixed and does not depend on $i$. Now, from previous step and the equation 
\begin{equation*}
(m_0 z_i)^2 \lambda_{i}^2=(m_1\lambda_1 z_1+\cdots+m_{n_0}\lambda_{n_0} z_{n_0})^2
\end{equation*}
we conclude $\{ \lambda_{i} \}$ does not have any accumulation point and since $\{\lambda_{i}\} \subset (0,1)$ it follows that it's finite. Passing to a subsequence again, we may assume $\lambda_{i}=\lambda=const$.

\end{proof}

\begin{proof}[Proof of Theorem \ref{thm-mainco}]

The proof is quite similar to proof of Theorem \ref{thm-main}, just replace $C_i$ with $2x_i$, and  Proposition \ref{prop:lambda} with Proposition \ref{prop:lambda2}.

\end{proof}

\bibliography{aomsample}
\bibliographystyle{aomalpha}

\newpage
\end{document}